\documentclass{mcom-l}

\usepackage{amssymb,amsfonts}

\usepackage{graphicx}

\usepackage[usenames,dvipsnames]{color}
\usepackage{subcaption,float}
\usepackage{cite}
\usepackage{algorithm,algpseudocode}

\newtheorem{theorem}{Theorem}[section]
\newtheorem{proposition}[theorem]{Proposition}

\theoremstyle{definition}

\newtheorem{example}[theorem]{Example}

\theoremstyle{remark}
\newtheorem{remark}[theorem]{Remark}

\numberwithin{equation}{section}

\newcommand{\mb}[1]{\ensuremath{\mathbf{#1}}}
\newcommand{\ip}[2]{\ensuremath{\langle #1,#2\rangle}}

\newcommand{\CC}{\mathbb C}

\newcommand{\RR}{\mathbb R}

\newcommand{\ii}{\mathfrak i}

\newcommand{\cH}{\mathcal H}
\newcommand{\cL}{\mathcal L}
\newcommand{\cK}{\mathcal K}

\newcommand{\cB}{\mathcal B}
\newcommand{\cG}{\mathcal G}

\newcommand{\Ran}{\mathrm{Range}}
\newcommand{\Nul}{\mathrm{Nullspace}}

\newcommand{\argmin}{\mathrm{arg\,min}}
\newcommand{\Dom}{\mathrm{Dom}}
\newcommand{\Spec}{\mathrm{Spec}}
\newcommand{\Res}{\mathrm{Res}}
\newcommand{\dist}{\mathrm{dist}}
\newcommand{\curl}{\mathrm{curl}}

\begin{document}

\title[Eigenvectors near a subspace]{Detecting eigenvectors of an operator
	that are near a specified subspace}


\author[D. Darrow]{David Darrow} 
\address{David Darrow, Department of Mathematics, Massachussetts
	Institute of Technology, Cambridge, MA 02139}
\curraddr{}
\email{ddarrow@mit.edu}
\thanks{}

\author[J. Ovall]{Jeffrey S. Ovall} 
\address{Jeffrey S. Ovall,
	Fariborz Maseeh Department of Mathematics and Statistics, Portland
	State University, Portland, OR 97201}
\curraddr{}\email{jovall@pdx.edu}
\thanks{}

\subjclass[2010]{Primary }

\date{\today}

\dedicatory{}

\begin{abstract}
  In modeling quantum systems or wave phenomena, one is often
  interested in identifying eigenstates that approximately carry a
  specified property; scattering states approximately align with
  incoming and outgoing traveling waves, for instance, and electron
  states in molecules often approximately align with superpositions of
  simple atomic orbitals. These examples---and many others---can be
  formulated as the following eigenproblem: given a self-adjoint
  operator $\cL$ on a Hilbert space $\cH$ and a closed subspace
  $W\subset\cH$, can we identify all eigenvectors of $\cL$ that lie
  approximately in $W$?
	
  We develop an approach to answer this question efficiently, with a
  user-defined tolerance and range of eigenvalues, building upon
  recent work for spatial localization in diffusion
  operators~\cite{Ovall2023}. Namely, by perturbing $\cL$
  appropriately along the subspace $W$, we collect the eigenvectors
  near $W$ into a well-isolated region of the spectrum, which can then
  be explored using any of several existing methods. We prove key
  bounds on perturbations of both eigenvalues and eigenvectors,
  showing that our algorithm correctly identifies desired eigenpairs,
  and we support our results with several numerical examples.
\end{abstract}

\maketitle


\section{Introduction}\label{Introduction}
Understanding the eigenmodes of a linear operator is often key to understanding the physics it represents. In quantum mechanics, eigenmodes of the Hamiltonian correspond to the system's energy levels---for instance, the eigenproblem for an isolated Coulomb potential gives us the familiar breakdown of atomic orbitals into $s$, $p$, $d$, and $f$ types, along with their corresponding geometries. In fluid mechanics, eigenmodes (of positive eigenvalue) correspond to system instabilities, and their structure tells us the leading-order geometry of instability or turbulence.
In many applications, one is particularly interested in eigenstates that carry a specified property. Scattering problems---for instance, in quantum mechanics---involve searching for states that converge to radially-outgoing waves as they travel. In studying disordered materials, one may wish to identify the strongly-localized states created through \emph{Anderson localization}~\cite{Anderson1958,Lagendijk2009}. To understand the effect of a single defect in a crystal, however, one may instead search for eigenstates with only a localized perturbation away from a known, nonzero solution. These examples---and many others---can be formulated as the following eigenproblem: 
\begin{equation}
	\tag{T}\label{KeyTask}
	\parbox{\dimexpr\linewidth-4em}{%
		\strut
		Given a Hilbert space $\cH$, a self-adjoint operator $\cL$ on $\cH$ with compact resolvent, a closed subspace $W\subset\cH$, an interval $[a,b]$, and a tolerance $\delta^*>0$, find all eigenpairs $(\lambda,\psi)$ of $\cL$ such that $\lambda\in[a,b]$ and $\dist(\psi,W)/\|\psi\| \leq \delta^*$, or determine that none exist.%
		\strut
	}
\end{equation}

In this paper, we develop an efficient approach for determining such eigenpairs, extending a previous work by Ovall and Reid~\cite{Ovall2023}. Ovall and Reid tackled this question in the particular case of spatial localization in $\RR^d$; they hypothesized that $\cH = L^2(\Omega)$ for a bounded, connected open set $\Omega\subset\RR^d$, that $\cL$ is an elliptic operator on $\Omega$ with compact resolvent, and that $W = L^2(K)\subset L^2(\Omega)$ for a closed subdomain $K\subset\Omega$. Even in this restricted setting, their work carries important applications. For instance, much work has gone into understanding {Anderson localization}, wherein disorder can give rise to spatially-localized electron states in an otherwise uniform medium.
We mention some notable contributions in this direction: Filoche, Mayboroda and co-authors make extensive use of the \emph{landscape function}, the solution of the source problem $\cL u = 1$, to identify likely regions of localization and provide estimates of the associated groundstate energies and eigenvectors~\cite{Filoche2012,Arnold2016,Arnold2019,Arnold2019a}; Steinberger, Lu and Murphey provide an alternative to the landscape function that plays a similar role~\cite{Lu2018,Lu2021,Steinerberger2021}; Altmann, Henning and Peterseim apply a two-phase iterative approach to approximate
localized eigenstates~\cite{Altmann2019,Altmann2020}.  These approaches are well-suited to scenarios in which there are many localized eigenvectors
near the bottom of the spectrum, and they offer no \emph{a priori} control over how strongly localized a given eigenstate is (i.e., the tolerance $\delta^*$).
In contrast, the approach of Ovall and Reid is not constrained to the lower portion of the spectrum, can be used to detect localization in regions not predicted by the aforementioned approaches, and provides a natural mechanism for controlling the strength of localization.

In short, fixing a sufficiently small $s>0$, Ovall and Reid replace the given operator $\cL$ with a perturbed operator
\[\cL(s) = \cL + \ii s\,\chi_K\quad,\quad \Dom(\cL(s)) = \Dom(\cL)~,\]
with $\chi_K$ the characteristic function on $K\subset\Omega$, and they search for eigenpairs $(\mu,\phi)$ with $\Im\mu\approx s$. Morally, each such eigenmode should correspond to an eigenmode $\psi$ of $\cL$ highly localized within $K$, for which $\chi_K\psi \sim \psi$.  Indeed, they prove that, if $(\lambda,\psi)$ is an eigenpair of $\cL$ with $\lambda\in[a,b]$ and $\dist(\psi,K)/\|\psi\| < \delta^*$, there is necessarily an eigenpair $(\mu,\phi)$ of $\cL(s)$ with $\mu$ within a distance $s\delta^*$ of $[a,b] + \ii s$, provided that the parameter $s$ is not too large. Moreover, $(\Re\mu,\Re\phi)$ or $(\Re\mu,\phi)$ provides a good approximation of $(\lambda,\psi)$ that comes with built-in identities for the residual errors $\|(\cL-\Re\mu)\Re\phi\|$ and $\|(\cL-\Re\mu)\phi\|$.  The associated algorithm provides a powerful approach to address (\ref{KeyTask}) in their case of interest---particularly when one does not know \emph{a priori} how many eigenvalues are contained in $[a,b]$, and does not wish to use the ``brute force'' approach of computing them all and then testing the eigenvectors to see which satisfy the localization constraint.

In the present work, we extend their results to a more general environment, addressing (\ref{KeyTask}) for a general Hilbert space $\cH$, any self-adjoint operator $\cL$ with compact resolvent, and any closed subspace $W\subset\cH$. This generalization opens the door to several important applications: scattering problems and localized perturbations, as discussed above, but also localized eigenmodes of graph-theoretic operators, eigenmodes that obey approximate symmetries, and eigenmodes localized in a given band of linear or angular momenta.

Our theoretical results are proved in Section~\ref{Theory}, extending the work of Ovall and Reid to our more general setting, and a corresponding algorithm template is proposed in Section~\ref{Algorithm}. Section~\ref{Applications} is dedicated to exploring various applications of our work, establishing the mathematical objects appropriate to each setting and how to efficiently implement each. Accompanying these applications are several numerical examples, demonstrating the effectiveness of our algorithm in handling a variety of practical scientific problems.

\section{Theoretical Results}\label{Theory}
Let $\cK$ be a real Hilbert space, with inner-product
$\ip{\cdot}{\cdot}$ and norm, $\|\cdot\|$, and let $\cH$ denote its
complexification, which we identify with $\cK\oplus\ii\cK$.  Vectors
in $\cK$ are called real, and for each $v\in\cH$ we have the unique
decomposition into its real and imaginary parts, $v=\Re v+\ii\,\Im v$,
where $\Re v,\, \Im v\in\cK$. {Notably, any Hilbert space can be equipped with a (non-unique) decomposition $\cH = \cK\oplus\ii\cK$, starting from an orthonormal basis, so our construction does not restrict the choice of $\cH$.}

Now, let $\cL:\Dom(\cL)\to \cH$, $\Dom(\cL)\subset\cH$, be a self-adjoint
operator with compact resolvent. Among other properties, this
implies that $\cL$ has a discrete spectrum of real eigenvalues, and
that each eigenspace of $\cL$ is finite-dimensional.

\begin{example}[Reaction-Diffusion]\label{ReactionDiffusion}
	We draw attention to the example discussed
        in~\cite{Ovall2023}, with $\cH = L^2(\Omega;\CC)$ (so $\cK=L^2(\Omega;\RR)$) for a bounded open set $\Omega\subset\mathbb{R}^d$, and $\cL$ an elliptic operator on $\Omega$ of the form
	\[\cL v = -\nabla\cdot(A\nabla v) + cv~.\]
        We take $\Dom(\cL)=\{v\in H^1_0(\Omega;\CC):\, \cL v\in \cH\}$.
        Typical assumptions that ensure that $\cL$ is self-adjoint and
        has compact resolvent are that $A\in[L^\infty(\Omega;\RR)]^{d\times d}$ is
        postive definite with eigenvalues bounded below (almost
        everywhere) by a positive constant $\sigma>0$, and that 
        $c\in L^\infty(\Omega)$ is bounded below.
\end{example}

\begin{example}[Magnetic Schr\"odinger]\label{MagneticSchrodinger}
	Another example of central interest in quantum mechanics uses
        the same $\cH$ as above, and defines the operator
	\[\cL v = (-\ii\nabla-\mb{A})\cdot(-\ii\nabla-\mb{A})v+ c
          v~,\] together with the same domain used in the previous
        example.  This differential operator is related to the motion
        of a charged particle in a magnetic field.  With
        $d=2,3$, the {(real)} vector field $\mb{A}$ is
        called the magnetic potential, and $\curl(\mb{A})$
        represents the underlying magnetic field, up to some
        scaling.  We take the same assumptions on the scalar potential
        $c$ as in Example~\ref{ReactionDiffusion}.  Although we might
        relax the regularity requirements on $\mb{A}$ and still obtain
        a self-adjoint operator, it is sufficient to assume that $\mb{A}$ is $C^1$
        in each of its components.  When $c=0$, $\cL$ is referred to
        as the magnetic Laplacian, which has been studied extensively
        using the techniques of semi-classical analysis,
        cf.~\cite{Fournais2006a,Fournais2019,Fournais2023}.
        Asymptotics and numerical approximation of eigenvectors lower
        in the spectrum are considered in~\cite{BonnaillieNoel2006,
          Bonnaillie-Noel2007, BonnaillieNoel2016,
          BonnaillieNoel2016a}, and an exploration of how $\mb{A}$
        (and $c$)
        affects the spatial localization of such eigenvectors is taken
        up in~\cite{OvallArXiv2023,Poggi2024,Hoskins2024}.
\end{example}

For any operator $\cB:\Dom(\cB)\to\cH$, $\Dom(\cB)\subset \cH$, we denote its spectrum and resolvent
set by $\Spec(\cB)$ and $\Res(\cB)$, respectively.
For $\lambda\in \Spec(\cB)$, we define
$E(\lambda,\cB)=\Nul(\lambda-\cB)$ to be the corresponding
eigenspace.  For $\Lambda\subset\Spec(\cB)$, we denote the
corresponding invariant subspace by
$E(\Lambda,\cB)=\bigoplus\{E(\lambda,\cB):\,\lambda\in \Lambda\}$.
We say that $\cB$ is a \textit{real operator} when $\Re(\cB v)=\cB(\Re
v)$ for al $v\in\cH$.  This is equivalent to saying that $\cB v$ is
real whenever $v$ is real.
We note that the reaction-diffusion operators in Example~\ref{ReactionDiffusion} are
real operators, and the magnetic Schr\"odinger operators in
Example~\ref{MagneticSchrodinger} are not, unless $\mb{A}=\mb{0}$.

Let $W$ be a closed subspace of $\cH$, and let $Q:\cH\to\cH$ be the
$\cH$-orthogonal projector onto $W$, so $W=\Ran(Q)$.  
Given $v\in \cH$, $v\neq 0$, we define two complementary
measures of closeness to $W$:
\begin{align}\label{DeltaTau}
\tau(v)=\frac{\|Qv\|}{\|v\|}\quad,\quad \delta(v)=\frac{\|(I-Q)v\|}{\|v\|}~,
\end{align}
where we use $\|\cdot\|$ here and below to denote the norm of $\cH$. It is clear that $\tau$ and $\delta$ are scale-invariant, or
$\tau(cv)=\tau(v)$ and $\delta(cv)=\delta(v)$ for all $c\in\CC$, and
that $\tau^2(v)+\delta^2(v)=1$. {Furthermore, since $\|(I-Q)v\| = \dist(v,W)$, the normalized distance used in (\ref{KeyTask}) is exactly $\delta(v)$.}

\begin{proposition}\label{ProjectorIsReal}
  Let $Q$ be as defined above.  Then $Q$ is a real operator.
\end{proposition}
\begin{proof}
  In complex
  vector spaces of the type we consider, we have $\|v\|^2=\|\Re
  v\|^2+\|\Im v\|^2$ for all $v\in\cH$.  Suppose that $v$ is real, and
  let $w\in W$.  We have $\|v-w\|^2=\|v-\Re w\|^2+\|\Im w\|^2$, so
  $\|v-w\|\geq \|v-\Re w\|$, with equality iff $\Im w = 0$.  Since
  $\|v-Qv\|=\inf\{\|v-w\|:\,w\in W\}$, it follows that $Qv$ is real.
\end{proof}

Given $s\geq0$, we define $\cL(s)$ by
\begin{align*}
  \cL(s)\,v=\cL v+\ii\,s\,Qv\quad,\quad \Dom(\cL(s))=\Dom(\cL)~.
\end{align*}
In the terminology of Kato, $\cL(s)$ is a holomorphic family of Type
(A).  In fact, $\cL(s)$
may be viewed as a subfamily of the self-adjoint holomorphic family of
type (A), $\cL(z)=\cL+\ii\,z\,Q$ for $z\in\CC$, cf. Chapter VII, Sections 2 and 3 of~\cite{Kato1995}.

\begin{theorem}\label{KeyTheorem}
For any eigenspace $E(\lambda,\cL)$, there is an orthonormal basis $\{\psi_1,\ldots,\psi_m\}$ such that we can identify a family of
eigenpairs $(\mu_i(s),\phi_i(s))$ of $\cL(s)$ for which $\mu_i(s)\to\lambda$ and
$\phi_i(s)\to\psi_i$ as $s\to 0$.  For such a family, we have
\begin{align}\label{LimitingEigenvalue}
  \lim_{s\to 0}\frac{\lambda+\ii\,s-\mu_i(s)}{\ii\,s}=\delta^2(\psi_i)~.
\end{align}
It follows that, for $s$ sufficiently small,
\[\dist(\lambda+\ii\,s,\Spec(\cL(s)))\leq \,\min_{1\leq i\leq m}\delta(\psi_i) \leq s\,D~,\]
where $D = \max_{\psi\in E(\lambda,\cL)}\delta(\psi)$. In particular, if $(\lambda,\psi)$ is a simple eigenpair of $\cL$, then for $s$ sufficiently small,
	\[\dist(\lambda + \ii\,s,\Spec(\cL(s)))\leq s\,\delta(\psi).\]
\end{theorem}
\begin{proof}
	The identification of eigenpairs $(\mu_i(s),\phi_i(s))$ of $\cL(s)$
	converging to $E(\lambda,\cL)$ follows from~\cite[Chapter VII,
	Theorems 1.7, 1.8]{Kato1995}, and we can designate $\psi_i=\phi_i(0)$. 
	
        To see that we can take $\psi_i$ to be orthonormal, note
        from~\cite[Chapter VII, Theorem 3.9, Remark 3.10]{Kato1995}
        that for each $\mu_i(s)$, there is a complex neighborhood
        $U_i$ of $0$ such that $s\mapsto\mu_i(s)$ is holomorphic for
        $s\in U_i$. Since $E(\lambda,\cL)$ is finite-dimensional, all
        $\mu_i(s)$ are holomorphic in $U = \bigcap U_i$. As such, we
        can find a sufficiently small set $U'\subset U$ such that any
        two of $\mu_i(s)$ either (a) agree on all of $U'$ or (b) only agree at $s=0$. Consider the
        self-adjoint perturbations
	\[\cL(\ii\,s) = \cL - s\,Q~,\]
	with associated eigenpairs
        $(\mu_i(\ii\,s),\phi_i(\ii\,s))$. Since $\cL(\ii\,s)$ is
        self-adjoint, any pair of subspaces $E(\mu_i(\ii\,s), \cL(\ii\,s))$
        is either identical or orthogonal for all $s$. 
        As $s\to 0$, however, these must
        converge to the same subspaces as $E(\mu_i(s), \cL(s))$, and $\psi_i$ can be
        constructed as any orthonormal basis of each limiting
        subspace.
	
	Next, for a fixed $s\in\RR$, we calculate
  \begin{align*}
    \mu_i(s)(\phi_i(s),\psi_i)&=(\cL(s)\phi_i(s),\psi_i)\\
                        &=(\cL\phi_i(s),\psi_i)+\ii\,s(Q\phi_i(s),\psi_i)\\
                        &=\lambda(\phi_i(s),\psi_i)+\ii\,s(\phi_i(s),\psi_i)-\ii\,s\,((I-Q)\phi_i(s),\psi_i)~.
  \end{align*}
	Rearranging, it holds that
	\begin{align}\label{BasicIdentity}
		\frac{\lambda+\ii\,s-\mu_i(s)}{\ii\,s}=\frac{((I-Q)\phi_i(s),\psi_i)}{(\phi_i(s),\psi_i)}~,
	\end{align}
	taking $s$ small enough that $(\phi_i(s),\psi_i)\neq 0$. 
 The result~\eqref{LimitingEigenvalue} then follows directly
 from~\eqref{BasicIdentity}.  The final assertion is an immediate
 consequence of~\eqref{LimitingEigenvalue}.
\end{proof}

{In fact, we have the following lower bound on how small $s$ needs to be in order to guarantee}
the eigenvalue closeness
result $\dist(\lambda+\ii\,s,\Spec(\cL(s)))\leq s\delta(\psi)$.

\begin{theorem}\label{EncodingTheorem}
  Let $d=\dist(\lambda,\Spec(\cL)\setminus\{\lambda\})$ and
  $D=\max\{\delta(\psi): \psi\in E(\lambda,\cL)\setminus\{0\}\}$.  If
  $|s|<d/2$, then
  \begin{align*}
    \dist(\lambda+\ii\,s,\Spec(\cL(s)))\leq
    \frac{|s| D^2}{1-2|s|/d}~.
  \end{align*}
 It follows that $\dist(\lambda+\ii\,s,\Spec(\cL(s)))\leq
  s\,D$ when $|s|\leq (1-D)d/2$.
\end{theorem}
\begin{proof}
  By shifting the operator, $\cL\longleftarrow \cL+c$ for some
  $c\in\RR$, if necessary, we may assume that $|\lambda|>d$ without
  loss of generality. 
  Suppose that $|s|<\frac{d}{2}$.  By\cite[Chapter VII, Theorem 2.6,
  Remark 2.9]{Kato1995}, all of the eigenvalues $\mu_i(s)$ of $\cL(s)$
  associated with $\lambda$ lie in the disk $B(\lambda,|s|)$, and no
  other eigenvalues of $\cL(s)$ do.  Choose one such $\mu(s)$ and let
  $\phi(s)$ be an associated eigenvector of $\cL(s)$.  Let $P$ denote
  the spectral projector onto $E(\lambda,\cL)$.  Since $\cL$ is
  self-adjoint, $P$ itself is an orthogonal projection, and
  $P\cL=\cL P$.  Using this commuting property, and the idempotency of
  $P$, we have
\begin{align}\label{EncodingTheoremA}
I-P=((\cL+\ii\,s)(I-P)-\mu(s))^{-1}(I-P) (\cL+\ii\,s-\mu(s))~.
\end{align}
We note that
$\Spec((\cL+\ii\,s)(I-P))=\{0\}\cup\Spec(\cL+\ii\,s)\setminus\{\lambda+\ii\,s\}$,
and our assumptions guarantee that
$\dist(\mu(s), \Spec((\cL+\ii\,s)(I-P)))\geq d/2$, so the resolvent
appearing in~\eqref{EncodingTheoremA}.  is well-defined.  By direct
computation, we see that
\begin{align}\label{EncodingTheoremB}
  (I-P)\phi(s)=\ii\,s ((\cL+\ii\,s)(I-P)-\mu(s))^{-1}(I-P)(I-Q)\phi(s)~,
\end{align}
from which it follows that
\begin{align}\label{EncodingTheoremC}
  \|(I-P)\phi(s)\|\leq \frac{|s|\,\|(I-Q)\phi(s)\|}{\dist(\mu(s),\Spec((\cL+\ii\,s)(I-P)))}\leq a\|(I-Q)\phi(s)\|\quad,\quad a=2|s|/d~.
\end{align}
Since $\|(I-Q)\phi(s)\|\leq
\|(I-Q)P\phi(s)\|+\|(I-Q)(I-P)\phi(s)\|\leq
\|(I-Q)P\phi(s)\|+a\|(I-Q)\phi(s)\|$, we determine that
\begin{align}\label{EncodingTheoremD}
  \|(I-P)\phi(s)\|\leq \frac{a}{1-a}\, \|(I-Q)P\phi(s)\|~.
\end{align}

Now, $P\phi(s)$ is a non-zero element of
$E(\lambda,\cL)$, so the argument establishing~\eqref{BasicIdentity}
for $\psi_i$ applies directly to $P\phi(s)$.  More specifically, we have
\begin{align*}
  \frac{\lambda+\ii\,s-\mu(s)}{\ii\,s}&=\frac{((I-Q)\phi(s),P\phi(s))}{(\phi(s),P\phi(s))}\\
&=\frac{((I-Q)P\phi(s),P\phi(s))+((I-Q)(I-P)\phi(s),P\phi(s))}{(\phi(s),P\phi(s))}\\
&=\frac{\|(I-Q)P\phi(s)\|^2+((I-Q)(I-P)\phi(s),P\phi(s))}{\|P\phi(s)\|^2}\\
&=\delta^2(P\phi(s))+\frac{((I-Q)(I-P)\phi(s),P\phi(s))}{\|P\phi(s)\|^2}~.
\end{align*}
Combining this identity with~\eqref{EncodingTheoremD}, we determine that
\begin{align*}
  \frac{|\lambda+\ii\,s-\mu(s)|}{|s|}&\leq\delta^2(P\phi(s))+\frac{a}{1-a}\,\delta^2(P\phi(s))
                                       =\frac{\delta^2(P\phi(s))}{1-a}\leq \frac{D^2}{1-a}~,
\end{align*}
which is equivalent to the first assertion of the theorem.  The second
assertion of the theorem follows directly from the first, by
determining $a$ for which $D/(1-a)\leq 1$.

We briefly consider the two extreme cases, $D=1$ and
$D=0$.  If $D=0$, this means that $Q\psi=\psi$ for all $\psi\in E(\lambda,\cL)$,
so $\cL(s)\psi=(\lambda+\ii\,s)\psi$, and we have 
$\dist(\lambda+\ii\,s,\Spec(\cL(s)))=0$ regardless of $s$.  If
$D=1$, then we are only considering the case $s=0$, in
which case we also clearly have
$\dist(\lambda+\ii\,s,\Spec(\cL(s)))=0$.
So, for any $D$, the final assertion holds.
\end{proof}


Theorems~\ref{KeyTheorem} and~\ref{EncodingTheorem} concern how to
``encode'' both $\lambda$ and $\delta(\psi)$ for an eigenpair
$(\lambda,\phi)$ of $\cL$ into an eigenvalue of $\cL(s)$ when $s$ is
sufficiently small.  The theorem below can be considered a
complementary ``decoding'' result in the sense that it concerns obtaining
eigenpairs of $\cL$ from eigenpairs of $\cL(s)$.  Although several
aspects of the proof
are very similar to those given in~\cite[Theorem 2.4,
Proposition 2.6]{Ovall2023}, we
include an argument here for completeness.

\begin{theorem}\label{KeyTheorem2}
Let $(\mu,\phi)$ be an eigenpair of $\cL(s)$, and let
$\delta=\delta(\phi)$ and $\tau=\tau(\phi)$.  Then $\Im\mu=s\,\tau^2$ and
we have the eigenvalue closeness bounds
\begin{align}\label{EigenvalueCloseness2}
s\, \delta^2\leq \dist(\mu,\Spec(\cL+\ii\,s))\leq s\,
  \delta\quad,\quad
  \dist(\Re\mu,\Spec(\cL))\leq s\, \delta\, \tau~.
\end{align}
Furthermore, we have residual identity
\begin{align}\label{RealPartResidual1}
\|(\cL-\Re\mu)\phi\|&=s\,\delta\,\tau\, \|\phi\|~.
\end{align}
If $\cL$ is a real operator, then we also have the residual identity
\begin{align} \label{RealPartResidual2}
 \|(\cL-\Re\mu)(\Re\phi)\|^2&=s^2\left(\tau^4\|(I-Q)(\Im
                                           \phi)\|^2+\delta^4\|Q(\Im
                                           \phi)\|^2\right)~.
\end{align}
Finally, let $\lambda=\argmin\{|\sigma-\Re\mu|:\,\sigma\in\Spec(\cL)\}$.
If $\Lambda\subset\Spec(\cL)$ contains $\lambda$, then
\begin{align}\label{EigenvectorCloseness2}
\inf_{v\in
  E(\Lambda,\cL)}\frac{\|\phi-v\|}{\|\phi\|}\leq
  \frac{s\,\delta\,\tau}{\dist(\Re\mu,\Spec(\cL)\setminus(\Lambda\cup\{\Re\mu\}))}~,
\end{align}
where $E(\Lambda,\cL)$ is the invariant subspace of $\cL$ associated
with the eigenvalues $\Lambda$.
\end{theorem}
\begin{proof}
  The fact that $\Im\mu = s\tau^2$ follows immediately from the
  Rayleigh quotient $\mu=(\cL(s)\phi,\phi)/(\phi,\phi)$.  Combining
  this with the fact
  that the eigenvalues of $\cL+\ii\,s$ have $s$ as their imaginary
  part establishes lower bound in the first inequality
  of~\eqref{EigenvalueCloseness2}.

  In what follows, it will be convenient to take $\mu_1=\Re\mu$,
  $\mu_2=\Im\mu$, $\phi_1=\Re\phi$ and $\phi_2=\Im\phi$.
  If $\mu \in\Spec(\cL+\ii\,s)$, then $\tau^2=1$, so $\delta=0$, and the
  inequality $\dist(\mu,\Spec(\cL+\ii\,s))\leq s\, \delta$ holds
  trivially.  So we assume that $\mu\notin \Spec(\cL+\ii\,s)$.  Since
 $(\cL+\ii\,s-\mu)\phi=\ii\,s\,(I-Q)\phi$, it follows that $1\leq
 \|(\cL+\ii\,s-\mu)^{-1}\|\,s\delta$.  Noting that $\cL+\ii\,s$ is
 normal, we have
 $\|(\cL+\ii\,s-\mu)^{-1}\|^{-1}=\dist(\mu,\Spec(\cL+\ii\,s))$, which
 establishes the first of the upper bounds in~\eqref{EigenvalueCloseness2}.
  
Simple manipulation of the identity
$(\cL+\ii\,s-\mu)\phi=\ii\,s\,(I-Q)\phi$ reveals that
\begin{align*}
  (\cL-\mu_1)\phi=\ii\,s\tau^2\,(I-Q)\phi-\ii\,s\delta^2\,Q\phi~.
\end{align*}
Taking norms and using the Pythagorean theorem, we see that
\begin{align*}
  \|(\cL-\mu_1)\phi\|^2/\|\phi\|^2&=s^2\left(\tau^4\delta^2+\delta^4\tau^2\right)=s^2\, \delta^2\,\tau^2~,
\end{align*}
which establishes~\eqref{RealPartResidual1}.  Essentially the same
reasoning used to establish the first of the upper bounds
in~\eqref{EigenvalueCloseness2} can be used to establish $\dist(\Re\mu,\Spec(\cL))\leq s\, \delta\, \tau$.
When $\cL$ is a real operator, we also have the identity
\begin{align*}
  (\cL-\mu_1)\phi_1=s\delta^2\,Q\phi_2-s\tau^2\,(I-Q)\phi_2~,
\end{align*}
from which~\eqref{RealPartResidual2} follows.

Now let $P=\frac{1}{2\pi\,\ii}\int_\gamma(z-\cL)^{-1}\,dz$ denote the
(orthogonal) spectral projector for $E(\Lambda,\cL)$, where $\gamma$
is a simple closed contour that encloses $\Lambda\cup\{\mu_1\}$, and
excludes $\Spec(\cL)\setminus\Lambda$.  Noting that $P$ commutes with $\cL$,
the identity
\begin{align}\label{ComplementaryProjectorIdentity2}
I-P=(\cL(I-P)-\mu_1)^{-1}(I-P)(\cL-\mu_1)~,
\end{align}
follows by direct algebraic manipulation.  Since
$\mu_1=(\cL\phi,\phi)/(\phi,\phi)\geq \lambda_{\min}(\cL)>0$, we know
that $\mu_1\not\in\Spec(\cL(I-P))$, so
\begin{align*}
\|(I-P)\phi\|\leq
  \|(\cL(I-P)-\mu_1)^{-1}\|\|(\cL-\mu_1)\phi\|\leq \frac{s\,\delta\,\tau\,\|\phi\|}{\dist(\Re\mu,\Spec(\cL)\setminus(\Lambda\cup\{\Re\mu\}))}~,
\end{align*}
which establishes~\eqref{EigenvectorCloseness2}.  
\end{proof}


\begin{remark}[Canonical Rescaling]\label{EigenvectorRotation}
  In the case of a real operator $\cL$, such as reaction-diffusion
  operators in Example~\ref{ReactionDiffusion}, the corresponding
  eigenvectors can be taken to be real, and the residual
  identity~\eqref{RealPartResidual2} suggests a canonical rescaling of
  $\phi$.  For any non-zero $c\in\CC$, taking $\tilde\phi=c\phi$, we have
  \begin{align*}
    \|(\cL-\Re\mu)(\Re\tilde\phi)\|_{L^2(\Omega)}^2&=s^2\left([\tau(\tilde\phi)]^4\|(I-Q)(\Im
                                           \tilde\phi)\|_{L^2(\Omega)}^2+[\delta(\tilde\phi)]^4\|Q(\Im
                                                     \tilde\phi)\|_{L^2(\Omega)}^2\right)\\
   &=s^2\left([\tau(\phi)]^4\|(I-Q)(\Im
                                           \tilde\phi)\|_{L^2(\Omega)}^2+[\delta(\phi)]^4\|Q(\Im
                                                     \tilde\phi)\|_{L^2(\Omega)}^2\right) ~.
  \end{align*}
  For the second equality, we used the scale-invariance of $\delta$
  and $\tau$.  Determining $c$ that minimizes this (squared) residual
  can be recast as finding the real eigenvector corresponding to the
  smallest eigenvalue of a $2\times 2$ positive semidefinite matrix.
  This determines the canonical rescaling $c$ up to (real) sign.
\end{remark}

\section{Algorithm Template}\label{Algorithm}
We saw in the preceding section that one can extend the theoretical
results of Ovall and Reid~\cite{Ovall2023} to a far wider class of
problems---our Hilbert space $\cH$ need not be a function space on
$\RR^d$, our operator $\cL$ need not be an elliptic differential
operator, and our subspace $W\subset\cH$ need not encode spatial
localization. In this section, we leverage these highly general results to propose a highly general computational algorithm, suitable for all problems of the form (\ref{KeyTask}):
\begin{equation}
	\tag{T}\label{KeyTask2}
	\parbox{\dimexpr\linewidth-4em}{%
		\strut
		Given a Hilbert space $\cH$, a self-adjoint operator $\cL$ on $\cH$ with compact resolvent, a closed subspace $W\subset\cH$, an interval $[a,b]$, and a tolerance $\delta^*>0$, find all eigenpairs $(\lambda,\psi)$ of $\cL$ such that $\lambda\in[a,b]$ and $\dist(\psi,W)/\|\psi\| \leq \delta^*$, or determine that none exist.%
		\strut
	}
\end{equation}

\begin{algorithm}
	\caption{Eigenproblem with Approximate Linear Constraints}\label{ELAT}
	\begin{algorithmic}[1]
		\Procedure{constrained\_eigensolver}{$a,b,s,\delta^*,W$}
		\State find all eigenpairs $(\mu,\phi)$ of $\cL(s)$ with $\dist(\mu,[a,b]+\ii) \leq s\, \delta^*$\label{alg1}
		\Comment{Thm.~\ref{KeyTheorem}}
		\If{no eigenvalues found}
		\State exit
		\Comment{No constrained eigenpairs exist}
		\Else
		\For{each $(\mu,\phi)$ found}
		\State normalize $\phi\longleftarrow \Re(c\phi)$ if desired \label{alg1rot}
		\Comment{Rmk.~\ref{EigenvectorRotation}}
		\State post-process $(\Re\mu,\phi)$ to obtain eigenpair $(\tilde{\lambda},\tilde\psi)$ of $\cL$ \label{alg2}\Comment{Thm.~\ref{KeyTheorem2}}
		\If{$\delta(\tilde\psi)\leq\delta^*$ and $\tilde{\lambda}\in[a,b]$}\label{alg3}
		\State accept $(\tilde\lambda,\tilde\psi)$
		\EndIf
		\EndFor
		\EndIf
		\State \textbf{return} accepted eigenpairs $(\tilde\lambda,\tilde\psi)$
		\EndProcedure
	\end{algorithmic}
\end{algorithm}

This algorithm template is a natural generalization of that presented
in~\cite{Ovall2023}, and it takes the same form.  We highlight that
the template does not constrain the method used to discretize
differential operators such as the reaction-diffusion operator or the
magnetic Schr\"odinger operator in Examples~\ref{ReactionDiffusion}
and~\ref{ReactionDiffusion}, the method used to solve the
(discretized) eigenvalue problem in line~\ref{alg1}, or the method
used for post-processing (if necessary) in line~\ref{alg2}.  Suitable
choices for each may depend on the structure of $\cH$ and $\cL$.   We
comment on this further in Section~\ref{Applications}.

In both~\cite{Ovall2023} and the present contribution, we consider
elliptic differential operators as an important motivating
application, and use a finite element discretization of the eigenvalue
problem for $\cL(s)$ (and for $\cL$).  The finite element space
consists of globally continuous, piecewise polynomial functions of
local degree $p$ on a relatively fine triangulation of the domain
$\Omega$, with triangles having characteristic edge length $h$.
The finite element
software package NGSolve~\cite{Schoeberl2014,Schoeberl2021} is used
for discretization, visualization and several other aspects of the
computation.

For solving the discrete (matrix) eigenvalue problem associated with
line~\ref{alg1}, any number of eigenvalue solvers might be employed,
provided they can efficiently and reliably return eigenpairs whose
eigenvalues are in some specified region of the complex plane---in our
case, the region within $s\delta^*$ of the segment $[a,b]+\ii$.
In~\cite{Ovall2023}, an implementation of FEAST eigensolver that is
integrated with NGSolve~\cite{Gopalakrishnan2017} was used.
FEAST~\cite{Polizzi2009} is a popular member of a broader family of
eigensolvers~\cite{Sakurai2007,Ikegami2010,Beyn2012,Kestyn2016,Huang2018,Yin2019a}
that can be derived from (quadrature) approximations of contour
integrals of the resolvent along the boundary of the region of
interest.  In our experiments, we found it convenient to use the
\texttt{eigs} command of MATLAB, which employs Krylov subspace
methods~\cite{Lehoucq1998,Stewart2001/02} to compute a few eigenpairs , with parameter settings to
select eigenvalues having largest imaginary part.

Theorem~\ref{KeyTheorem2} provides residual identities for approximate
eigenpairs $(\Re\mu,\phi)$ or $(\Re\mu,\Re\phi)$ of $\cL$ that were
computed in the initial phase of the algorithm (line~\ref{alg1}).
If the residual is small enough, there is no need for post-processing
(line~\ref{alg2}), which was the case in our experiments.  In cases
where post-processing is desired, a natural choice is to use
a small number of steps (perhaps only one step) of shifted inverse
iteration on (a discrete version of) $\cL$, with shift $\sigma=\Re\mu$ and initial eigenvector guess
$\phi$ (or $\Re\phi$).  Alternatively, one might employ a Krylov
subspace method that permits initial guesses for the eigenpair of interest.

We highlight several possible applications in the following section,
and illustrate some of them with numerical experiments.

\section{Applications and Numerical Illustrations}\label{Applications}
Ovall and Reid~\cite{Ovall2023} presented an earlier version of our algorithm in the specific case of \emph{position-space localization in} $\RR^d$. That is, given a domain $\Omega\subset\RR^d$, a differential operator $\cL$ on $\Omega$, and a closed subdomain $K\subset\Omega$, they demonstrated how to use the present techniques to identify eigenvectors of $\cL$ approximately supported in $K$.

Our generalized framework allows us to see similar benefits in a more general range of problems. {For one, there are clear analogues of their spatial localization in alternate bases, which are straightforwardly handled by our algorithm: eigenstates with momentum, spin, or angular momentum in a certain range, or more abstractly, eigenfunctions of differential operators with a certain subspace of basis functions. Along with this simple case, we outline a number of possible applications below, with numerical examples for several. Note that this list is far from exhaustive. }

\subsection{Localized Perturbations.}\label{app:perturbations}
Suppose our eigenfunctions of interest are not exactly localized within a given region $K$, but we know how they are supposed to behave outside that region. For instance, we may seek to characterize how a given electronic state of a crystal changes near a local defect, or how the propagation of gravity waves changes over an underwater obstacle. 

Note that localized perturbations are \emph{not} covered by the framework of Ovall and Reid~\cite{Ovall2023}. Their work considers the case where the \emph{full} eigenstate is approximately localized within a given region. Conversely, we here consider localized perturbations away from a known, (generically) non-zero state; this situation might model a defect in a crystal or large molecule, or a moving object on the surface of a body of water.

Suppose that $\cH$ is a space of functions on $\Omega$. Also suppose that we have a known function $u\in \cH$, and that our desired eigenstates satisfy
\begin{equation}\label{eq:localpert}
	\psi|_{\Omega\setminus K}\approx c\, u|_{\Omega\setminus K}~,
      \end{equation}
for some constant $c$.
For instance, if we have a well-understood operator $\mathcal{L}_0$ that agrees with $\mathcal{L}$ outside of $K$, we might choose $u$ to be a fixed eigenmode of $\mathcal{L}_0$; our algorithm would then tell us how $u$ reacts to a local (but potentially large!) perturbation of the system.

To isolate states of the form (\ref{eq:localpert}), we define the subspace
\[W = \{\psi\in\cH:\psi|_{\Omega\setminus K}\in \mathrm{span}\{u|_{\Omega\setminus K}\}\}\]
and the projection
\begin{equation}\label{eq:localpert_proj}
	Q_K:\cH\to W,\;\psi(\mathbf{r})\mapsto \chi_{K}(\mathbf{r})\,\psi(\mathbf{r}) + \chi_{\Omega\setminus K}(\mathbf{r})\,u(\mathbf{r})\frac{\int_{\Omega\setminus K}u^*\psi\,d\mb{r}}{\int_{\Omega\setminus K}u^*u\,d\mb{r}}~.
\end{equation}
Applying our algorithm would reveal all eigenstates of $\cL$ sufficiently close to $W$, or, equivalently, satisfying (\ref{eq:localpert}) with a fixed tolerance.

{The projection \eqref{eq:localpert_proj} is a rank-one update of the projection $\chi_K$ employed in~\cite{Ovall2023}, and it naturally gives rise to
  a rank-one update of the perturbed operator from that work, namely
  \begin{align*}
    \cL(s) = \cL + \ii s\,Q_K = (\cL + \ii s\,\chi_K) + \ii s\,\tilde{Q}_K~,
  \end{align*}
  where $\tilde{Q}_K:=Q_K-\chi_K$.  Upon  finite element discretization, for example, the discrete operator corresponding to $\cL + \ii s\,\chi_K$ is sparse,
  but that corresponding to $\ii s\,\tilde{Q}_K$ is not.  The latter is of rank one, however, so it can be dealt with efficiently.  More specifically, if our 
  eigensolver only requires us to apply the operator $\cL(s)$, then we can carry this procedure out by separately applying the terms $(\cL + \ii s\,\chi_K)$ and $\ii s\,\tilde{Q}_K$.  On the level of discretization, both products entail computational cost that is linear with respect to the size of the discretization.  If our eigenvalue solver requires us to apply the resolvent of $\cL(s)$, we can apply the Sherman--Morrison formula~\cite{Sherman1950} using the same ``sparse plus low-rank'' splitting.}

In general, we may know only that $\psi|_{\Omega\setminus K}$ is close to the span of a finite number of functions $\{u_1,...,u_n\}$. In this case, we only need update (\ref{eq:localpert_proj}) to include terms for each $u_k$, and to replace the Sherman--Morrison formula with the general Woodbury formula~\cite{Woodbury1950} if our eigensolver requires application of the resolvent.

An example of this application is shown in Figure~\ref{fig:square1}. Here, we let $\Omega$ be the square $(-1,1)\times(-1,1)\subset\RR^2$, and we consider the Schrödinger operator $\cL = -\Delta-18\pi^2\chi_{\Omega/2}$ with Neumann boundary conditions. If the second term of the Schrödinger operator---which carves out a small square well in the center of the domain---were missing, we would find the ground state to be $u = \text{const.}$, and we might expect one or more eigenstates of the perturbed operator to maintain this structure outside the well. 

In this problem, we search for eigenstates of $\cL$ that \emph{approximately} match $u$ in the exterior region $\Omega\setminus(\Omega/2)$; to do so, we discretize our domain with a mesh length $h = 0.05$ and a polynomial degree $3$, and we apply our algorithm with $s=0.1$ and a tolerance $\tau^2\geq 0.9$. Out of the first $100$ eigenvectors of $\cL(s)$ (ordered by real part), only $9$ are accepted: $\{\phi_0,...,\phi_7, \phi_{13}\}$. To demonstrate the concept, we show eigenvectors $\phi_{12}$, $\phi_{13}$, and $\phi_{16}$ using two colorbars; the first is a standard colorbar, and the second shows whether the function is greater than, less than, or in between $[-0.25,0.25]$. 

As a baseline, none of the three eigenvectors are well-localized in $\Omega/2$; applying the algorithm of Ovall and Reid~\cite{Ovall2023} gives 
\[\tau^2(\phi_{12},\Omega/2)=0.84\quad,\quad \tau^2(\phi_{13},\Omega/2)=0.47\quad,\quad\tau^2(\phi_{16},\Omega/2)=0.60~,\]
using their notation, all below our tolerance of $0.9$. 

However, we might expect (from the second colorbar) that $\phi_{13}$ is approximately constant in $\Omega\setminus(\Omega/2)$. Indeed, our algorithm gives
\[\tau^2(\phi_{13})=0.96~,\]
above our chosen threshold. The other eigenvalues' $\tau^2$ values are unchanged from before, suggesting that they have vanishing overlap with $u$. Moreover, scaling our eigenvectors as $\phi\longleftarrow c\,\phi$ to minimize the residual (\ref{RealPartResidual2}) yields the values
\begin{gather*}
	\|(\cL-\Re\mu)(\Re\phi_{12})\| = 1.55\times 10^{-5}~,\\
	\|(\cL-\Re\mu)(\Re\phi_{13})\| = 8.21\times 10^{-5}~,\\
	\|(\cL-\Re\mu)(\Re\phi_{16})\| = 1.88\times 10^{-4}~,
\end{gather*}
giving a strong justification for not doing any post-processing at all.

\begin{figure}
	\begin{subfigure}[c]{0.08\textwidth}
		\centering
		\caption*{$ $}
		\includegraphics[scale=0.45]{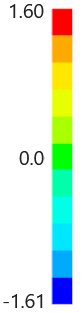}
	\end{subfigure}%
	\begin{subfigure}[c]{0.28\textwidth}
		\centering
		\caption*{$\phi_{12}$: Rejected}
		\includegraphics[scale=0.42]{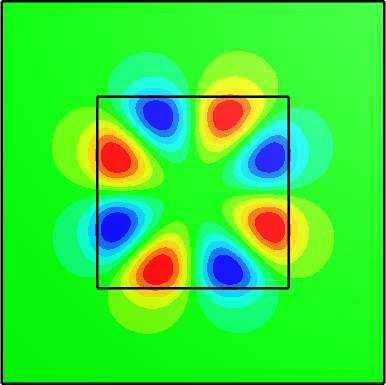}
	\end{subfigure}%
	\begin{subfigure}[c]{0.28\textwidth}
		\centering
		\caption*{$\phi_{13}$: Accepted}
		\includegraphics[scale=0.42]{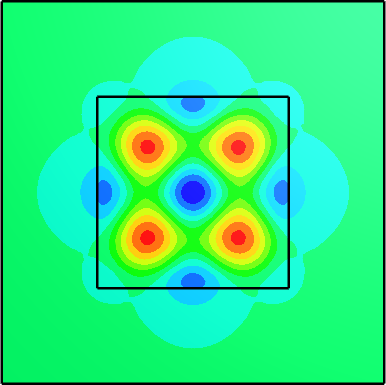}
	\end{subfigure}%
	\begin{subfigure}[c]{0.28\textwidth}
		\centering
		\caption*{$\phi_{16}$: Rejected}
		\includegraphics[scale=0.42]{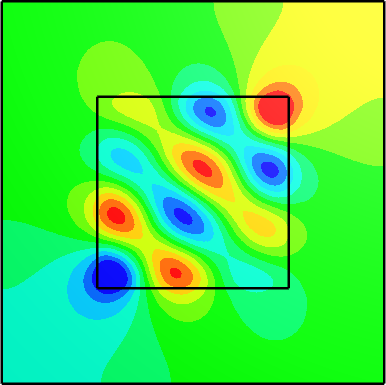}
	\end{subfigure}%
	\begin{subfigure}[c]{0.08\textwidth}
		\centering
		$ $
	\end{subfigure}%
	$ $
	\bigskip
	$ $
	\begin{subfigure}[c]{0.08\textwidth}
		\centering
		\includegraphics[scale=0.45]{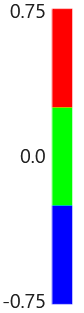}
	\end{subfigure}%
	\begin{subfigure}[c]{0.28\textwidth}
		\centering
		\includegraphics[scale=0.42]{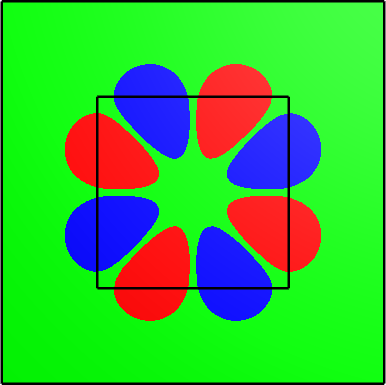}
	\end{subfigure}%
	\begin{subfigure}[c]{0.28\textwidth}
		\centering
		\includegraphics[scale=0.42]{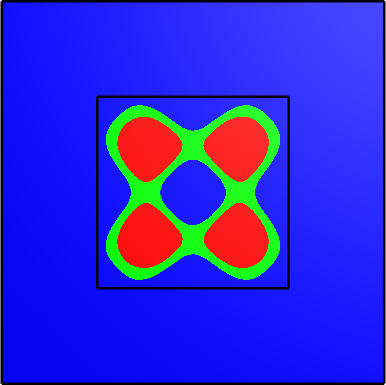}
	\end{subfigure}%
	\begin{subfigure}[c]{0.28\textwidth}
		\centering
		\includegraphics[scale=0.42]{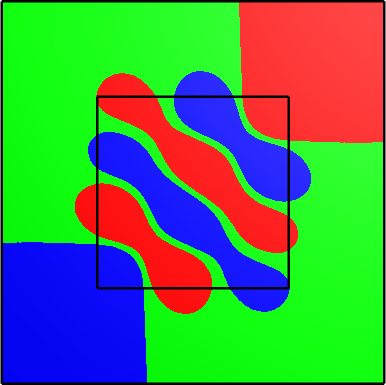}
	\end{subfigure}%
	\begin{subfigure}[c]{0.08\textwidth}
		\centering
		$ $
	\end{subfigure}%
	\caption{Eigenvector response to a strong, localized perturbation, as discussed in Section \ref{app:perturbations}. We show approximate eigenvectors calculated on the domain $\Omega=(-1,1)\times(-1,1)$ for the Schrödinger operator $\cL = -\Delta - 18\pi^2\chi_{\Omega/2}$, with Neumann boundary conditions. Here, we apply our algorithm to align eigenvectors with the subspace $W\subset L^2(\Omega)$ of functions that are constant outside $\Omega/2$. We show eigenvectors $\phi_{12}$, $\phi_{13}$, and $\phi_{16}$ above, using two different colorbars; $\phi_{12}$ and $\phi_{16}$ were rejected by our algorithm, while $\phi_{13}$ was accepted. The second colorbar above highlights the difference; even though $\phi_{13}$ is not localized inside the square well, it is approximately constant outside. Notably, all three eigenvectors are rejected by the localization algorithm of Ovall and Reid~\cite{Ovall2023}.}\label{fig:square1}
\end{figure}

\subsection{Approximate Symmetries.}\label{app:symmetry}
In the discussion of Section \ref{app:perturbations}, we know the approximate profile of our eigenmode in a localized region. However, our framework also extends naturally to partial information of other forms; for instance, if we know the eigenmode \emph{approximately} obeys a certain symmetry. For instance, we may want to identify approximately spherically symmetric electron states surrounding an asymmetric model of the nucleus, approximately translation-invariant electron states in a crystal with impurities, or instabilities in vortices that carry a certain number of approximately-identical nodes.

Fix a locally compact group $\cG$, and suppose it has the property that its Haar measure $dg$ is both left- and right-invariant. This property is known as \emph{unimodularity}, and applies to abelian groups, compact groups, discrete groups, semisimple Lie groups, and others~\cite{Bourbaki2004}.

Further, suppose $\cG$ carries a unitary representation $\cG\to U(\cH)$ on our Hilbert space $\cH$, and define the mean over $\cG$ as follows:
\[Q_\cG:\psi\mapsto\int_\cG g\psi\, dg~,\]
normalizing $\int_\cG dg = 1$.
Since the group action is unitary, the operator $Q_\cG$ is exactly the orthogonal projection onto the space of $\cG$-invariant vectors. Indeed, it is clear that the image of $Q_\cG$ is $\cG$-invariant, since
\[\tilde{g} Q_\cG\psi = \frac{1}{|\cG|}\int_\cG \tilde{g}g\psi\, dg = \frac{1}{|\cG|}\int_\cG \tilde{g}g\psi\, d(\tilde{g}g) = Q_\cG\psi\]
for any $\tilde{g}\in\cG$. We can also see that it is an orthogonal projection, because if $u\in\cH$ is orthogonal to the image of $Q_\cG$, then for any $v\in\cH$, we find
\[\langle v,Q_\cG u\rangle = \frac{1}{|\cG|}\int_\cG\langle v, gu\rangle\,dg = \frac{1}{|\cG|}\int_\cG\langle g^{-1}v, u\rangle\,dg = \frac{1}{|\cG|}\int_\cG\langle g^{-1}v, u\rangle\,dg^{-1} = 0,\]
using the fact that $dg = dg^{-1}$ for unimodular groups~\cite{Pap2002}. 

Thus, we can apply our algorithm to this choice of projection. In doing so, we recover eigenvectors of our operator $\cL$ that are approximately fixed by the action of $\cG$, or---physically---that approximately obey the symmetry quantified by $\cG$.

The numerical implementation of our algorithm depends somewhat on the details of the action $\cG\to U(\cH)$. For instance, suppose $\cH$ is a function space on $\Omega\subset\RR^d$, and $\cG$ is a small, finite group that permutes the points of $\Omega$. In this case, the projection
\[Q_\cG = \int_\cG g\;dg\]
is sparse, as it couples each point in $\Omega$ to only $O(|\cG|)$ others; we can thus apply and invert the total operator $\cL$ as usual. 

If $\cG$ is large, however, the projection may no longer be sparse. For instance, suppose $\cG = SO(d)$ acts on $\Omega\subset\RR^d$ by rotations; then $Q_\cG$ is the projection onto the set of radial functions, which is low-rank. We can then apply or invert the total operator by first splitting $\cL(s)$ into its sparse component $\cL$ and its low-rank component $Q_\cG$, as discussed for localized perturbations and wave scattering.

\begin{figure}
	\begin{subfigure}[c]{0.08\textwidth}
		\centering
		\caption*{$ $}
		\includegraphics[scale=0.45]{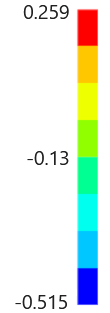}
	\end{subfigure}%
	\begin{subfigure}[c]{0.28\textwidth}
		\centering
		\caption*{$\phi_{13}$: Accepted}
		\includegraphics[scale=0.42]{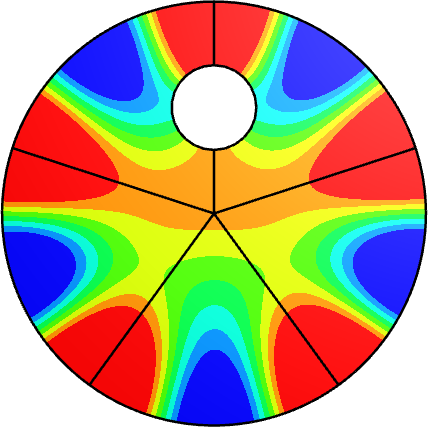}
	\end{subfigure}%
	\begin{subfigure}[c]{0.28\textwidth}
		\centering
		\caption*{$\phi_{32}$: Accepted}
		\includegraphics[scale=0.42]{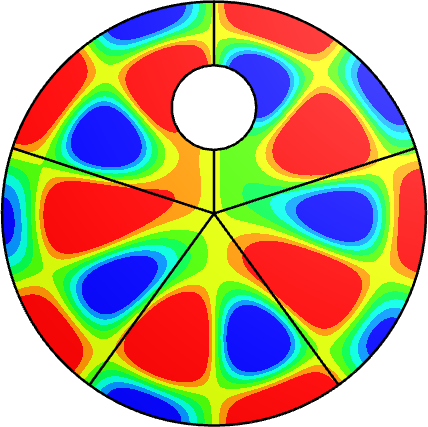}
	\end{subfigure}%
	\begin{subfigure}[c]{0.28\textwidth}
		\centering
		\caption*{$\phi_{41}$: Accepted}
		\includegraphics[scale=0.42]{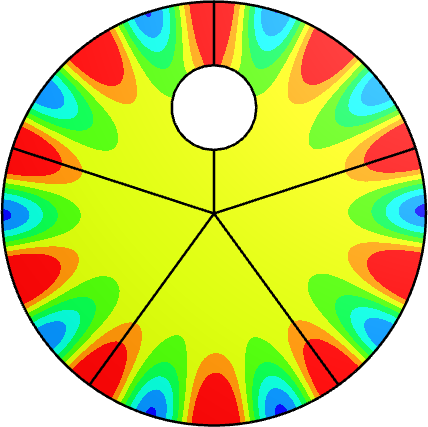}
	\end{subfigure}%
	\begin{subfigure}[c]{0.08\textwidth}
		\centering
		$ $
	\end{subfigure}%
	$ $
	\bigskip
	$ $
	\begin{subfigure}[c]{0.08\textwidth}
		\centering
		\caption*{$ $}
		\includegraphics[scale=0.45]{Figs/eig_test2/cbar1_new.png}
	\end{subfigure}%
	\begin{subfigure}[c]{0.28\textwidth}
		\centering
		\caption*{$\phi_{16}$: Rejected}
		\includegraphics[scale=0.42]{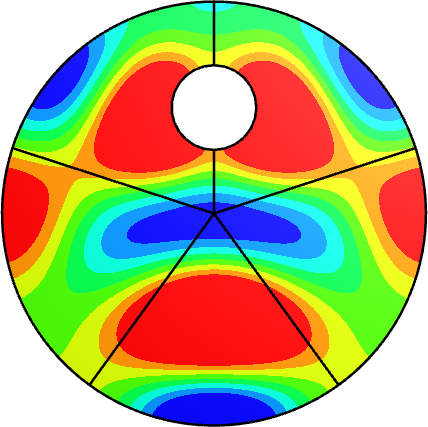}
	\end{subfigure}%
	\begin{subfigure}[c]{0.28\textwidth}
		\centering
		\caption*{$\phi_{33}$: Rejected}
		\includegraphics[scale=0.42]{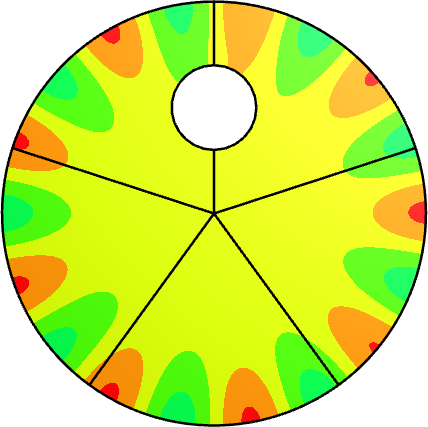}
	\end{subfigure}%
	\begin{subfigure}[c]{0.28\textwidth}
		\centering
		\caption*{$\phi_{52}$: Rejected}
		\includegraphics[scale=0.42]{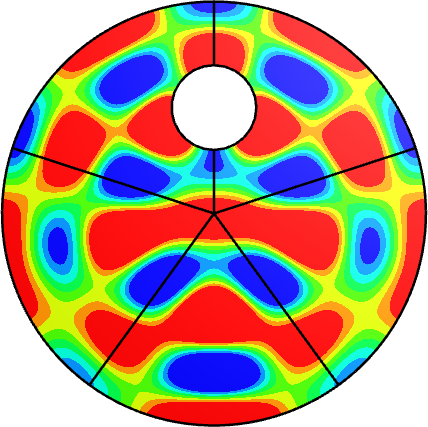}
	\end{subfigure}%
	\begin{subfigure}[c]{0.08\textwidth}
		\centering
		$ $
	\end{subfigure}%
	\caption{Eigenvectors with an approximate symmetry, as discussed in Section \ref{app:symmetry}. We show approximate eigenvectors of the Laplacian calculated on the domain $\Omega=B_1(0,0)\setminus B_{1/5}(0,1/2)$, with Neumann boundary conditions. Here, we apply our algorithm to align eigenvectors with the subspace $W\subset L^2(\Omega)$ of functions with five-point rotational symmetry. The top three eigenvectors, $\phi_{13}$, $\phi_{32}$, and $\phi_{41}$, are accepted by our algorithm, while the bottom three, $\phi_{16}$, $\phi_{33}$, and $\phi_{52}$, are rejected. Of the bottom three eigenvectors, only $\phi_{52}$ shows a trace of five-point symmetry, and this symmetry breaks down near the deleted region $B_{1/5}(0,1/2)$.}\label{fig:yen1}
\end{figure}

An example of this application is shown in Figure~\ref{fig:yen1}.  We write $B_r(x,y)\subset\RR^2$ for the ball of radius $r$ about a point $(x,y)\in\RR^2$. In this problem, we consider the eigenproblem for the Laplacian $\cL = -\Delta$ on the domain $\Omega = B_1(0,0)\setminus B_{1/5}(0,1/2)$, with Neumann boundary conditions. {While eigenvectors on the disk $B_1(0,0)$ take a familiar form, as products of radial Bessel functions and angular sinusoids, the deleted region near the top of the disk complicates this story.}

We apply our algorithm to find eigenvectors with an (approximate) five-point rotational symmetry, discretizing our domain with a mesh length $h = 0.05$ and a polynomial degree $3$, and using $s=0.1$ and a tolerance $\tau^2\geq 0.9$.  We approximate the subspace $W\subset L^2(\Omega)$ of functions with five-point rotational symmetry using Zernike functions~\cite{Zernike1934,Niu2022},
$Z^m_n(r,\theta)=R^{|m|}_n(r)e^{\ii\,m \theta}$, up to degree $n=15$. Concretely, we extend any function $f$ on $\Omega$ to a function $\tilde{f}$ on the unit disk, as $\tilde{f} = f\,\chi_\Omega$. Zernike functions are (a) exact eigenfunctions of the projector $Q:\cH\to W$ and (b) an orthogonal basis of $L^2$ on the unit disk, so we can exactly represent $W$ as a span of appropriate Zernike functions. For a more practical implementation, we recommend projecting onto $W$ by averaging a function over its orbit under the point group $C_5$.

Of the first $100$ eigenvectors of $\cL(s)$ (ordered by real part), only six are accepted: $\{\phi_0,\phi_{12},\phi_{13},\phi_{32},\phi_{40},\phi_{41}\}$. We show six example eigenvectors of $\cL(s)$. The top three, $\phi_{13}$, $\phi_{32}$, and $\phi_{41}$, show a clear (though imperfect) five-point symmetry. As such, they are accepted by our algorithm, with 
\[\tau^2(\phi_{13})=0.98\quad,\quad \tau^2(\phi_{32})=0.90\quad,\quad\tau^2(\phi_{41})=1.00~.\]
The bottom three eigenvectors, $\phi_{16}$, $\phi_{33}$, and $\phi_{52}$, are rejected by our algorithm, with
\[\tau^2(\phi_{16})=0.32\quad,\quad \tau^2(\phi_{33})=0.00\quad,\quad\tau^2(\phi_{52})=0.72~.\]
These are rejected for three different reasons. The eigenvector $\phi_{16}$ carries a nonzero $\tau^2$ value because of its nonzero \emph{constant} component, but it is dominated by a three-point symmetry orthogonal to $W$. The eigenvector $\phi_{33}$ carries a strong eighteen-point symmetry, so is (approximately) exactly orthogonal to $W$. The final eigenvector, $\phi_{52}$, carries several artifacts of a five-point symmetry, but it is dominated by fine structure at the scale of the missing region $B_{1/5}(0,1/2)$; its symmetry is thus strongly broken near that region, and it falls under our threshold of $0.9$.

We report the residuals of each eigenvector below, normalizing as discussed in Remark~\ref{EigenvectorRotation}:
\begin{gather*}
	\|(\cL-\Re\mu)(\Re\phi_{13})\| = 1.17\times 10^{-4}~,\\
	\|(\cL-\Re\mu)(\Re\phi_{16})\| = 3.58\times 10^{-4}~,\\
	\|(\cL-\Re\mu)(\Re\phi_{32})\| = 1.61\times 10^{-4}~,\\
	\|(\cL-\Re\mu)(\Re\phi_{33})\| = 6.92\times 10^{-5}~,\\
	\|(\cL-\Re\mu)(\Re\phi_{41})\| = 2.85\times 10^{-4}~,\\
	\|(\cL-\Re\mu)(\Re\phi_{52})\| = 1.08\times 10^{-4}~.
\end{gather*}
Notably, the residuals left by accepted eigenvectors are not significantly larger or smaller than rejected eigenvectors; both sets represent legitimate (approximate) eigenvectors of $\cL$, with the only difference being their distance to $W\subset\cH$.

\subsection{Localization in Alternate Bases}\label{app:localization}
Abstractly, we can view the position-space localization problem as a localization ``along'' the position operator 
\[\mathbf{Q}:\psi(\mathbf{q})\mapsto \mathbf{q}\,\psi(\mathbf{q})~,\]
defined element-wise by pairwise-commuting maps
\[Q_i:\psi(q_1,...,q_d)\mapsto q_i\,\psi(q_1,...,q_d)~.\]
If $\phi$ is a shared eigenvector of all $Q_i$ with corresponding eigenvalues $q_i$, we say that $(\mathbf{q},\phi)=((q_1,...,q_d),\phi)$ is an eigenpair of $\mathbf{Q}$. This notation is widely used in quantum mechanics---e.g., see Griffiths' text~\cite{Griffiths2017}.

If an eigenvector $\psi$ of $\cL$ is localized in a region $K\subset\Omega\subset\RR^d$, we can equivalently think of it as supported in the subspace
\[W=\operatorname{span}\{\phi\,|\,(\mathbf{q},\phi)\,\text{is an eigenpair of}\,\mathbf{Q}\,\text{with}\,\mathbf{q}\in K\}~.\]
Of course, if we are truly interested in the function space $L^2(\Omega)$, this construction is purely heuristic (though widely used); such eigenmodes exist only in a space of distributions, and $\psi$ would be constructed as an integral over an uncountable number of them.

More rigorous constructions can be carried out in alternate bases, however. In general, suppose we have an (unbounded) operator
\[\mathbf{B}:\Dom(\mathbf{B})\to\cH^d~,\]
for instance, constructed from a collection of pairwise commuting elements 
\[B_i:\Dom(B_i)\to\cH\quad,\quad B_i\in\{B_1,...,B_d\}~,\]
and suppose we fix a region $K\subset\CC^d$ to localize about. We construct the following subspace of $W$:
\[W=\operatorname{span}\{\phi\,|\,(\mathbf{b},\phi)\,\text{is an eigenpair of}\,\mathbf{B}\,\text{with}\,\mathbf{b}\in K\}~.\]
In general, if $\mathbf{B}$ does not commute with $\cL$, one cannot expect eigenvectors of the latter to lie exactly in $W$. However, we can apply our algorithm to identify eigenvectors that lie \emph{near} to $W$.

\begin{figure}
	\begin{subfigure}[c]{0.08\textwidth}
		\centering
		\caption*{$ $}
		\includegraphics[scale=0.45]{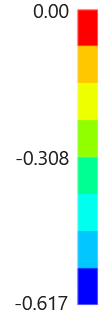}
	\end{subfigure}%
	\begin{subfigure}[c]{0.28\textwidth}
		\centering
		\caption*{$\phi_{8}$: Accepted}
		\includegraphics[scale=0.42]{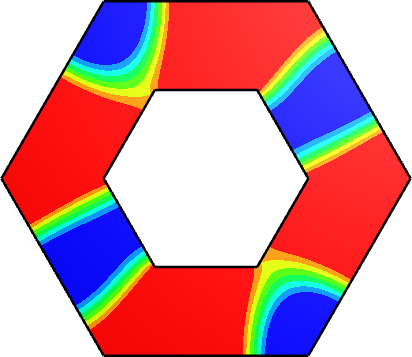}
	\end{subfigure}%
	\begin{subfigure}[c]{0.28\textwidth}
		\centering
		\caption*{$\phi_{9}$: Rejected}
		\includegraphics[scale=0.42]{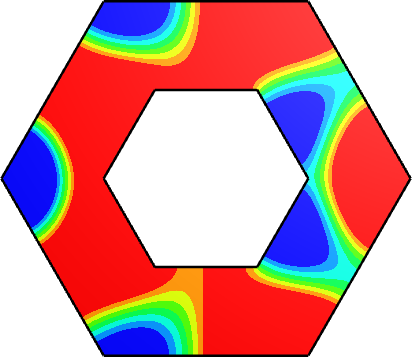}
	\end{subfigure}%
	\begin{subfigure}[c]{0.28\textwidth}
		\centering
		\caption*{$\phi_{10}$: Rejected}
		\includegraphics[scale=0.42]{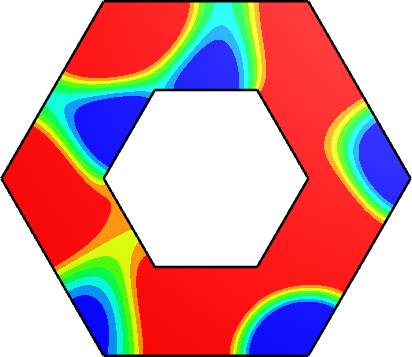}
	\end{subfigure}%
	\begin{subfigure}[c]{0.08\textwidth}
		\centering
		$ $
	\end{subfigure}%
	$ $
	\bigskip
	$ $
	\begin{subfigure}[c]{0.08\textwidth}
		\centering
		\includegraphics[scale=0.45]{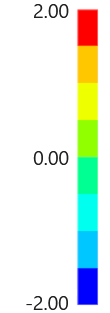}
	\end{subfigure}%
	\begin{subfigure}[c]{0.28\textwidth}
		\centering
		\includegraphics[scale=0.42]{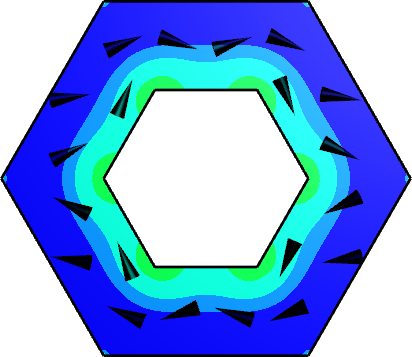}
	\end{subfigure}%
	\begin{subfigure}[c]{0.28\textwidth}
		\centering
		\includegraphics[scale=0.42]{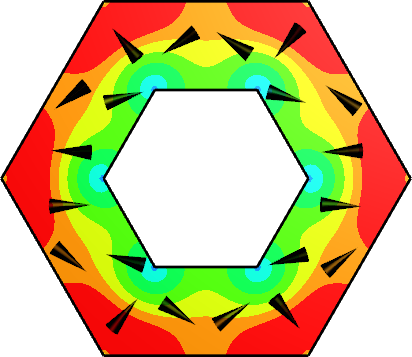}
	\end{subfigure}%
	\begin{subfigure}[c]{0.28\textwidth}
		\centering
		\includegraphics[scale=0.42]{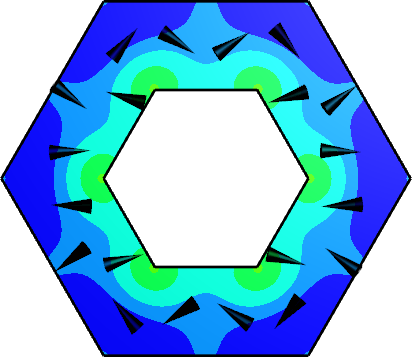}
	\end{subfigure}%
	\begin{subfigure}[c]{0.08\textwidth}
		\centering
		$ $
	\end{subfigure}%
	\caption{Localization in a basis of angular momentum states, as discussed in Section \ref{app:localization}. Specifically, we show approximate eigenvectors of the Laplacian on the hexagonal annulus shown, with Neumann boundary conditions; we apply our algorithm to align eigenvectors with the subspace $W\subset L^2(\Omega)$ of functions with strictly negative (i.e., clockwise) angular momentum. We show eigenvectors $\psi_{8}$, $\psi_{9}$, and $\psi_{10}$, along with their corresponding angular momenta---the arrowheads on the latter plots indicate direction of probability current. Only the first eigenvector is accepted by our algorithm, reflecting its uniformly negative angular momentum.}\label{fig:hex1}
\end{figure}

An example of this application is given in Figure \ref{fig:hex1}. Let $U$ be the regular hexagon inscribed in the unit circle, and define the hexagonal annulus $\Omega=U\setminus(U/2)$. We consider the Laplace operator $\cL=-\Delta$ on $\Omega$ with Neumann boundary conditions, and we search for eigenmodes of strictly negative angular momentum. More concretely, if 
\[\mathbf{B} = L_z = -\ii \left(x\partial_y - y\partial_x\right) = -\ii\,\partial_\theta\]
is the angular momentum operator in 2D, we define $W$ to be the span of eigenstates of $L_z$ of eigenvalue $ \leq-1$. 

We discretize the domain $\Omega$ with a mesh length $h = 0.05$ and a polynomial degree $3$. As discussed in Section \ref{app:symmetry}, we approximate the subspace $W\subset L^2(\Omega)$ of negative-momentum functions using Zernike functions, but now up to degree $n=8$. 
We apply our algorithm with $s=0.1$ and a tolerance $\tau^2\geq 0.55$ for accepted eigenstates. The lower tolerance in this experiment, as compared to the threshold $\tau^2\geq 0.9$ chosen in previous examples, reflects our approximation of $W$ using Zernike functions; because any function on $\Omega$ is truly interpreted as a function on the disk with \emph{support} in $\Omega$, its components along (low-dimensional) Zernike functions are artificially reduced. For instance, even the constant function $f(\mathbf{q})=1$ on $\Omega$ only has a component $\operatorname{vol}(\Omega)/\pi\approx0.62$ along the direction of the first Zernike function, $Z_0^0=1$. We use the latter value to scale our usual tolerance as $0.62\times0.9 \approx 0.55$.


Of the first $100$ eigenvectors of $\cL(s)$, only four are accepted: $\{\phi_{2},\phi_4,\phi_8,\phi_{15}\}$. Higher-energy eigenvectors \emph{may} be of negative angular momentum, but simply not be well-approximated by our restricted basis of $n\leq 8$ Zernike functions. We show three example eigenstates of $\cL(s)$, $\phi_8$, $\phi_9$, and $\phi_{10}$, using two plots each, in Figure~\ref{fig:hex1}. The upper row shows their real component $\Re\phi$, and the lower row shows their local angular momentum $\phi^*L_z\phi$, with arrowheads indicating the direction of the local probability current
\[\mathbf{J} = \Im\left(\phi^*\nabla\phi\right)~.\]
For reference, if $\rho = |\phi|^2$ is the Born probability density associated with a wavefunction $\phi$, then the pair $(\rho,\mathbf{J})$ satisfies the continuity equation
\[\partial_t\rho + \nabla\cdot\mathbf{J} = 0~,\]
prompting the definition of $\mathbf{J}$ as a quantum ``probability current''~\cite{Griffiths2017}.

The shown eigenstates carry $\tau^2$ values
\[\tau^2(\phi_{8})=0.79\quad,\quad \tau^2(\phi_{9})=0.37\quad,\quad\tau^2(\phi_{10})=0.47~.\]
Of these eigenstates, only $\phi_8$ is accepted by our algorithm; indeed, we can see visually that its angular momentum is strongly negative over most of the domain, and its probability current always travels clockwise. The remaining two eigenstates are related to one another by a reflection, both containing regions of non-negative angular momentum (and counterclockwise-facing $\mathbf{J}$) that push them away from $W$.

For this first experiment, we report residuals \emph{without} taking real parts of the approximate eigenvectors---i.e., using (\ref{RealPartResidual1}) rather than (\ref{RealPartResidual2}). This choice is made such that our selected solutions remain approximate eigenvectors of $L_z$; since the latter is not a real operator (in the sense that $\Re(L_zv)\neq L_z(\Re v)$), this property is incompatible with the scaling of Remark~\ref{EigenvectorRotation}. The corresponding residuals are
\begin{gather*}
	\|(\cL-\Re\mu)(\phi_{8})\| = 0.041~,\\
	\|(\cL-\Re\mu)(\phi_{9})\| = 0.048~,\\
	\|(\cL-\Re\mu)(\phi_{10})\| = 0.050~.
\end{gather*}

\subsection{Wave Scattering.}\label{app:scattering} 
A particularly important case of localized perturbations is given by \emph{scattering problems}, where we seek to identify how a plane wave scatters off of a localized obstacle or potential. Scattering problems have far-reaching applications in acoustics, electromagnetism, quantum chemistry, and fluid dynamics~\cite{Colton2013,Mandal2000}. For instance, Rutherford scattering off of atomic nuclei can be used as a spectroscopic technique, to diagnose the chemical composition of a material. Similarly, scattering of acoustic and radio waves is often used to determine the positions and velocities of distant objects---these processes gives rise to \emph{echolocation} and \emph{radiolocation} (radar), respectively.

Consider a potential $V\in L^\infty(\RR^d)$ well-localized about the origin, so that, say, $V$ falls off no slower than a Coulomb potential. Often just as important as understanding the \emph{bound} states of $V$ (which could be enumerated efficiently with traditional algorithms) is understanding its \emph{scattering amplitudes}; if a plane wave $\psi_i(\mathbf{r}) = e^{\ii\,\mathbf{k}\cdot\mathbf{r}}$ strikes the potential $V$, it generically gives rise to a spherical outgoing wave
\[\psi_s(\mathbf{r}) = f(\xi)\,r^{(1-d)/2}e^{\ii\,kr}~,\]
where $r = \|\mathbf{r}\|$, $k = \|\mathbf{k}\|$, and $\xi = \mathbf{r}/r\in S^{d-1}$. Here, $f(\xi)$ is the \emph{scattering amplitude}, which gives rise to the more-familiar \emph{(differential) cross section}:
\[\frac{d\sigma}{d\Omega}(\xi) = |f(\xi)|^2\quad,\quad \sigma = \int_{S^{d-1}}|f(\xi)|^2\,d\xi~.\]


We can apply our present work to efficiently find the form of $\psi_s$, and thus the differential cross section. Namely, our scattering state corresponds to a stationary solution $\psi$ of our wave equation with the following far-field (i.e., $r\to\infty$) behavior:
\begin{equation}\label{eq:scattering_dist}
	\psi(\mathbf{r}) \approx \psi_i(\mathbf{r}) + \psi_s(\mathbf{r}) = e^{\ii\,\mathbf{k}\cdot\mathbf{r}} + f(\xi)\,r^{(1-d)/2}e^{\ii\,kr}~.
\end{equation}
Suppose we carve out a sphere of radius $R$ about the origin, select an orthogonal basis $\{f_j\}_{j\geq 1}$ of $L^2(S^{d-1})$, and form the functions
\[u_0 = e^{\ii\,\mathbf{k}\cdot\mathbf{r}}\quad,\quad u_j(\mathbf{r}) = f_j(\xi)\,r^{(1-d)/2}e^{\ii\,kr}~.\]
Following our earlier analysis, then, we define the projection
\begin{equation}\label{eq:scattering_proj}
	Q_R:\psi(\mathbf{r})\mapsto \chi_{r<R}(\mathbf{r})\,\psi(\mathbf{r}) + \chi_{r\geq R}(\mathbf{r})\sum\nolimits_{j\leq N} u_j(\mathbf{r})\frac{\int_{r>R}u_j^*\psi}{\int_{r>R}u_j^*u_j}~,
\end{equation}
where the sum is taken over $u_j$ up to a chosen threshold $N = N(k)$. The projection $Q_R$ maps functions to the space of ``perfect'' scattering states, or those which take exactly the form (\ref{eq:scattering_dist}) for $r>R$. Applying our algorithm to the wave operator $\cL$ with the projection $Q_R$, we thus recover the eigenstates of $\cL$ that approximately fit this form.

If we fix a constant threshold $N$ in (\ref{eq:scattering_proj}), the corresponding operator $Q_R - \chi_{r<R}$ ends up being low-rank after a spatial discretization, and we can implement an eigensolver for $Q_R$ as discussed in Section \ref{app:perturbations}. If no such threshold is chosen, the operator $Q_R - \chi_{r<R}$ is still relatively low-rank, and the same technique can be applied. For instance, if we expand our Hilbert space in a basis of spherical polynomials, we expect $O(n^d)$ basis elements in total, for a chosen degree $n\geq 1$, but only $O(n^{d-1})$ basis elements $f_j$ for the space of functions on the sphere $S^{d-1}$. Since the scattering states $u_j$ are parameterized by these basis elements, the resulting matrix is of rank $O(n^{d-1})$. 

As a note, the same procedure can be applied to scattering off of sharp obstacles. If our obstacle occupies a region $K\subset\RR^d$, we simply choose a radius $R\gg\operatorname{diam}(K)$ and use the same projection as before.

To be clear, the proposed method does not directly improve upon existing, specialized algorithms for scattering problems. In particular, scattering problems are often handled using integral equation methods~\cite{ColtonIntegral2013}, which do not need to discretize the full $d$-dimensional space in the first place. However, the most efficient methods are often specialized to particular dimensions or geometries~\cite{Hyde2005,Ambikasaran2015,Gillman2015,Hoskins2020}. By contrast, our approach may offer benefits because of its flexibility; for instance, it applies equally well to any geometry and dimension count, and it can be used to recover a large number of scattering states simultaneously, with minimal additional computational cost. This flexibility may be particularly useful for scattering off of anisotropic potentials---if we append appropriate plane waves to the sum in (\ref{eq:scattering_proj}), we can collect results for all incoming angles simultaneously. 



\subsection{Spatial Localization on a Graph.}\label{app:graph}
In a different direction, note that our generalized framework allows us to study even spatial localization in a wider variety of cases. Specifically, if we take $\cH$ to be a self-adjoint operator---either the Laplacian or a Schrödinger operator, for instance---on a graph, we can apply our algorithm to identify eigenvectors that are approximately localized on any subgraph. For instance, this allows us to study Andersen localization in a more general setting. This case closely mirrors the Euclidean case of Ovall and Reid~\cite{Ovall2023}, so we do not discuss it further here.

\subsection{Counterexamples to a Known Pattern.}\label{app:counterexample}
Finally, we may have a sense of how \emph{most} eigenvectors of $\cL$ behave, at least in a given interval of eigenvalues, and we may be interested to see if any examples subvert this pattern. For instance, electronic states of large molecules are often modeled as superpositions of atomic orbitals, centered on each component atom~\cite{Ching2012}. Given a certain molecule, we may wish to determine whether any low-energy states are \emph{not} closely approximated by this subspace.

For this, suppose we have a fixed, closed subspace $W\subset\cH$ that we wish to avoid, and form the projection $Q:\cH\to W$ as usual. When we apply our algorithm as usual, any eigenvectors $\psi$ of $\cL$ within a distance $\dist(\psi,W)/\|\psi\| < \delta^*$ correspond to eigenpairs $(\mu,\phi)$ of $\cL(s)$ with $\Im\mu \geq s(1-\delta^*)$, in the sense of Theorems \ref{KeyTheorem}, \ref{EncodingTheorem}, and \ref{KeyTheorem2}. To avoid these, we simply search the \emph{remaining} space $\Im\mu < s(1-\delta^*)$.

As a note, one could formulate this as an application of our \emph{unmodified} algorithm by replacing $W$ by $W^\perp$, $Q$ by $I-Q$, and $\delta^*$ by $1-\delta^*$.


\section{Concluding Remarks}\label{Conclusions}
We have extended the theoretical and numerical results of Ovall and Reid~\cite{Ovall2023} to a large class of constrained eigenproblems. Moreover, we have proposed a number practical applications of the described algorithm---ranging from capturing approximate symmetries to solving localization problems on discrete graphs---and provided numerical implementations supporting several of these example applications. In future work, we intend to explore the numerical behavior of our full algorithm, and to extend our analysis to a wider class of eigenproblems. We believe that the present work provides a powerful approach to solve eigenproblems under approximate, linear constraints, and that it might support a wide range of questions in scientific computing.

The work of J.O. was partially supported by the NSF grant DMS 2208056 ``Computational Tools for Exploring Eigenvector Localization''.  This contribution was motivated by discussions between the authors initiated at the 2024 Simons Collaboration on the Localization of Waves Annual Meeting.

\bibliographystyle{amsalpha}
\bibliography{titles}

\end{document}